\documentclass[11pt]{article}
\usepackage{a4wide}
\usepackage{graphicx}
\usepackage{color}
\usepackage{amsmath}
\usepackage[mathcal]{eucal}
\usepackage{amsthm}
\usepackage{amssymb}
\usepackage{caption}

\newtheorem{theo}{Theorem}[section]
\newtheorem{mydef}[theo]{Definition}
\newtheorem{prop}[theo]{Proposition}
\newtheorem{lem}[theo]{Lemma}
\newtheorem{cor}[theo]{Corollary}
\newtheorem{prob}[theo]{Problem}
\newtheorem*{mtheorem}{Theorem}

\newcommand{\qitem}[1]{\noindent\leavevmode\hangindent1.5\parindent%
  \noindent\hbox to1.5\parindent{#1\hss}\ignorespaces}

\newcommand\blfootnote[1]{%
  \begingroup
  \renewcommand\thefootnote{}\footnote{#1}%
  \addtocounter{footnote}{-1}%
  \endgroup
}
\newcommand{\col}{\mathrm{col}}
\newcommand{\wcol}{\mathrm{wcol}}
\newcommand{\dcol}{\mathrm{dcol}}

\newcommand{\edp}[2]{#1^{[\natural #2]}}
\newcommand{\epp}[2]{#1^{\natural #2}}

\usepackage{soul}

\title{Chromatic Numbers of Exact Distance Graphs\,\thanks{\,The research
    for this paper was started during a stay of the first two authors at
    the Mittag-Leffler Institute in Stockholm. JvdH and HAK would like to
    thank the Mittag-Leffler Institute for hospitality and support. DAQ
    thankfully acknowledges support from CONICYT, PIA/Concurso Apoyo a
    Centros Cient\'ificos y Tecnol\'ogicos de Excelencia con Financiamiento
    Basal AFB170001.}}

\author{Jan van den Heuvel,\thanks{\,Department of Mathematics, London
    School of Economics and Political Science, London WC2A 2AE, UK.} \ \
  H.A. Kierstead\,\thanks{\,School of Mathematical and Statistical
    Sciences, Arizona State University, Tempe, AZ 85287, USA.} \ \ and \
  Daniel A. Quiroz\,\footnotemark[2] \thanks{\,Current affiliation:
    Centro de Modelamiento Matem\'atico, Universidad de Chile, Santiago,
    Chile.}}

\date{}

\begin{document}
\maketitle

\blfootnote{\,Email: \texttt{j.van-den-heuvel@lse.ac.uk},
  \texttt{kierstead@asu.edu}, \texttt{dquiroz@dim.uchile.cl}.}

\begin{abstract}
  \noindent
  For any graph $G=(V,E)$ and positive integer~$p$, the \emph{exact
    distance-$p$ graph}~$\edp{G}{p}$ is the graph with vertex set~$V$,
  which has an edge between vertices~$x$ and~$y$ if and only if~$x$ and~$y$
  have distance~$p$ in~$G$. For odd~$p$, Ne\v{s}et\v{r}il and Ossona de
  Mendez proved that for any fixed graph class with \emph{bounded
    expansion}, the chromatic number of~$\edp{G}{p}$ is bounded by an
  absolute constant.

  Using the notion of \emph{generalised colouring numbers}, we give a much
  simpler proof for the result of Ne\v{s}et\v{r}il and Ossona de Mendez,
  which at the same time gives significantly better bounds. In particular,
  we show that for any graph~$G$ and odd positive integer~$p$, the
  chromatic number of~$\edp{G}{p}$ is bounded by the weak
  $(2p-1)$-colouring number of~$G$. For even~$p$, we prove that
  $\chi(\edp{G}{p})$ is at most the weak $(2p)$-colouring number times the
  maximum degree.

  For odd~$p$, the existing lower bound on the number of colours needed to
  colour $\edp{G}{p}$ when~$G$ is planar is improved. Similar lower bounds
  are given for $K_t$-minor free graphs.

\end{abstract}

\section{Introduction and Main Results}\label{sec1}

\subsection{Powers, exact powers, and exact distance graphs}

All graphs in this paper are assumed to be finite, undirected, simple and
without loops. For a graph $G=(V(G),E(G))$ (or just $(V,E)$ if the graph
under consideration is clear) and vertices $x,y\in V$, let $d_G(x,y)$
denote the distance between~$x$ and~$y$ in~$G$, that is, the number of
edges contained in a shortest path between~$x$ and~$y$.

For a positive integer~$p$, the \emph{$p$-th power graph $G^p=(V,E^p)$
  of~$G$} is the graph with~$V$ as its vertex set and~$E^p$ contains the
edge~$xy$ if and only if $d_G(x,y)\le p$. Problems related to the chromatic
number $\chi(G^p)$ of power graphs~$G^p$ were first considered by Kramer
and Kramer~\cite{KK69a,KK69b} in 1969 and have enjoyed significant
attention ever since. It is clear that for $p\ge2$ any power of a star is a
clique, and hence there are not many classes of graphs for which
$\chi(G^p)$ can be bounded by a constant. An easy argument shows that for a
graph~$G$ with maximum degree $\Delta(G)\ge3$ we have
\[\chi(G^p)\le 1+\Delta(G^p)\le
1+\Delta(G)\cdot\sum_{i=0}^{p-1}(\Delta(G)-1)^i\in
\mathcal{O}(\Delta(G)^p).\]
However, there are many classes of graphs for which it is possible to find
much better upper bounds. Recall that a graph~$G$ is \emph{$k$-degenerate}
if every subgraph of~$G$ contains a vertex of degree at most~$k$.

\begin{theo}[Agnarsson \&
  Halld\'orsson~\cite{AH03}]\label{largepow}\mbox{}\\*
  Let $k$ and $p$ be positive integers. There exists a constant $c=c(k,p)$
  such that for every $k$-degenerate graph $G$ we have
  $\chi(G^p)\le c\cdot\Delta(G)^{\lfloor p/2\rfloor}$.
\end{theo}

\noindent
In this result, the exponent on $\Delta(G)$ is best possible (see below).
In particular, $\chi(G^2)$ is at most linear in $\Delta(G)$ for planar
graphs~$G$. Wegner~\cite{W77} conjectured that every planar graph~$G$ with
$\Delta(G)\ge 8$ satisfies
$\chi(G^2)\le \bigl\lfloor\frac32\Delta(G)\bigr\rfloor+1$, and gave
examples that show this bound would be tight. The conjecture has attracted
considerable attention since it was stated in~1977. For more on this
conjecture we refer the reader to \cite{AEvdH13,KK08}.

In \cite[Section~11.9]{NOdM12}, Ne\v{s}et\v{r}il and Ossona de Mendez
define the notion of \emph{exact power graph}. Let $G=(V,E)$ be a graph
and~$p$ a positive integer. The \emph{exact $p$-power graph~$\epp{G}{p}$}
has~$V$ as its vertex set, and~$xy$ is an edge in~$\epp{G}{p}$ if and only
if there is in~$G$ a path of length~$p$ (i.e.\ with~$p$ edges) between the
vertices~$x$ and~$y$ (the path need not be a shortest path). Similarly,
they define the \emph{exact distance-$p$ graph $\edp{G}{p}$} as the graph
with~$V$ as its vertex set, and~$xy$ as an edge if and only if
$d_G(x,y)=p$. Since obviously
$E(\edp{G}{p})\subseteq E(\epp{G}{p})\subseteq E(G^p)$, we have
$\chi(\edp{G}{p})\le\chi(\epp{G}{p})\le\chi(G^p)$.

For planar graphs~$G$, Theorem~\ref{largepow} gives that the exact
$p$-power graphs $\epp{G}{p}$ satisfy
$\chi(\epp{G}{p})\in\mathcal{O}\bigl(\Delta(G)^{\lfloor p/2\rfloor}\bigr)$.
This result is best possible, even for outerplanar graphs, as the following
examples show. For $k\ge2$ and $p\ge4$, let $T_{k,\lfloor p/2\rfloor}$ be
the $k$-regular tree of radius $\bigl\lfloor\frac12p\bigr\rfloor$ with
root~$v$. We say that a vertex $z$ is at level $\ell$ if $d(v,z)=\ell$. For
every edge $xy$ between vertices at levels~$\ell$ and $\ell+1$ for some
$\ell\ge1$, we do the following: if~$p$ is even, then add a path of length
$\ell+1$ between~$x$ and~$y$; if $p$ is odd, then add paths of length
$\ell+1$ and $\ell+2$ between~$x$ and~$y$. Call the resulting
graph~$G_{k,p}$. It is straightforward to check that $\Delta(G_{k,p})\le2k$
for even $p$, that $\Delta(G_{k,p})\le3k$ for odd $p$, and that there is a
path of length~$p$ between any two vertices at level
$\bigl\lfloor\frac12p\bigr\rfloor$. Since there are $k(k-1)^{\lfloor
  p/2\rfloor-1}$ vertices at level $\bigl\lfloor\frac12p\bigr\rfloor$, this
immediately means that $\chi(\epp{G_{k,p}}{p})\ge k(k-1)^{\lfloor
  p/2\rfloor-1}\in \Omega(\Delta(G_{k,p})^{\lfloor p/2\rfloor})$.

Surprisingly, for exact distance graphs, the situation is quite different.
In that case we can prove that for planar graphs~$G$ and odd~$p$ we have
$\chi(\edp{G}{p})\in\mathcal{O}(1)$, while for even~$p$ we have
$\chi(\edp{G}{p})\in\mathcal{O}\bigl(\Delta(G)\bigr)$. These bounds are
actually special cases of the following more general results. We will
recall the concept of a \emph{graph class with bounded expansion} in the
next subsection.

\begin{theo}\label{thm3}\mbox{}\\*
  Let $\mathcal{K}$ be a class of graphs with bounded expansion.

  \smallskip
  \qitem{(a)}Let~$p$ be an odd positive integer. Then there exists a
  constant $C=C(\mathcal{K},p)$ such that for every graph $G\in\mathcal{K}$
  we have $\chi(\edp{G}{p})\le C$.

  \smallskip
  \qitem{(b)}Let~$p$ be an even positive integer. Then there exists a
  constant $C'=C'(\mathcal{K},p)$ such that for every graph
  $G\in\mathcal{K}$ we have $\chi(\edp{G}{p})\le C'\cdot\Delta(G)$.
\end{theo}

\noindent
We will give two proofs of part\,(a). The two proofs give incomparable
bounds. Also, both proofs are considerably shorter and provide better
bounds than the original proof of part\,(a) of Ne\v{s}et\v{r}il and Ossona
de Mendez \cite[Theorem~11.8]{NOdM12}. Theorem~\ref{thm3}\,(b) is new, as
far as we are aware.

As we showed above, if we consider exact powers instead of exact distance
graphs, then we need to use bounds involving~$\Delta(G)$ if we want to
bound $\chi(\epp{G}{p})$, even for odd~$p$ and if $G$ is planar. However,
by adding the condition that~$G$ has sufficiently large \emph{odd girth}
(length of a shortest odd cycle), $\chi(\epp{G}{p})$ can be bounded without
reference to $\Delta(G)$, for odd $p$. It follows from
Theorem~\ref{thm3}\,(a) that this is possible if the odd girth is at least
$2p+1$. This is because odd girth at least $2p+1$ guarantees that if there
is a path of length~$p$ between~$u$ and~$v$, then any shortest $uv$-path
has odd length. With some more care we can reprove the following.

\begin{theo}[Ne\v{s}et\v{r}il \& Ossona de Mendez
  {\cite[Theorem~11.7]{NOdM12}}] \label{thm4} \mbox{}\\*
  Let $\mathcal{K}$ be a class of graphs with bounded expansion and let~$p$
  be an odd positive integer. Then there exists a constant
  $M=M(\mathcal{K},p)$ such that for every graph $G\in\mathcal{K}$ with odd
  girth at least $p+1$ we have $\chi(\epp{G}{p})\le M$.
\end{theo}

\noindent
Theorem~\ref{thm3}\,(a) is quite surprising, since already for planar
graphs~$G$, the exact distance graphs $\edp{G}{p}$ can be very dense. To
see this, for $i\ge2$ let $L_i$ be obtained from the complete graph~$K_4$
by subdividing each edge $i-1$ times (i.e.\ by replacing each edge by a
path of length~$i$). For $k\ge1$, form $L_{i,k}$ by adding four sets of~$k$
new vertices to~$L_i$ and joining all $k$ vertices in the same set to one
of the vertices of degree three in~$L_i$. See Figure~\ref{four} for a
sketch of $L_{1,k}$.

\begin{figure}[h]
 \centering
 \medskip
 \captionsetup{justification=centering}
 \includegraphics[width=53mm]{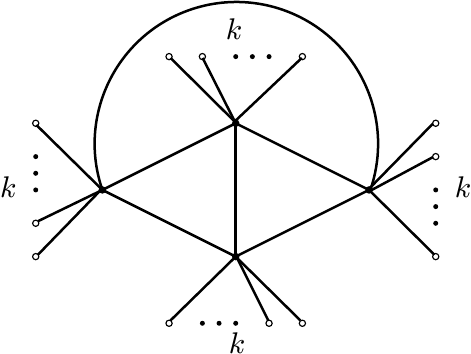}
 \smallskip
 \caption{ \ A graph $L_{1,k}$ such that $\edp{L_{1,k}}{3}$ has edge
   density approximately 3/4.}
  \label{four}
\end{figure}

It is easy to check that $L_{i,k}$ is a planar graph with $4+6(i-1)+4k$
vertices, while $\edp{L_{i,k}}{(i+2)}$ has $6k^2$ edges. So for fixed~$i$
and large~$k$, the graph $\edp{L_{i,k}}{(i+2)}$ has approximately $3/4$
times the number of edges of the complete graph on the same number of
vertices. Apart from having unbounded density, the graphs
$\edp{L_{i,k}}{(i+2)}$ have unbounded colouring number (and even unbounded
\emph{list chromatic number}), since $\edp{L_{i,k}}{(i+2)}$ contains a
complete bipartite graph $K_{k,k}$ as an (induced) subgraph. This makes the
fact that these graphs have bounded chromatic number even more surprising.

It is interesting to see what actual upper and lower bounds we can get for
the chromatic numbers of $\edp{G}{p}$ for~$G$ from some specific classes of
graphs and for specific values of (odd)~$p$. Using the proof
in~\cite{NOdM12}, it follows that for $p=3$ and for planar graphs~$G$ we
can get the upper bound $\chi(\edp{G}{3})\le5\cdot2^{20,971,522}$ (see also
Subsection~\ref{ssec3.0}). On the other hand, \cite[Exercise~11.4]{NOdM12}
gives an example of a planar graph~$G$ with $\chi(\edp{G}{3})=6$.

Our new proof of Theorem~\ref{thm3}\,(a) already gives a much smaller upper
bound for $\chi(\edp{G}{3})$ for planar graph~$G$. By a more careful
analysis, we can reduce that upper bound even further, giving the bound in
the following result. We also managed to increase the lower bound, although
by one only. Details can be found in Section~\ref{sec4}.

\begin{theo}\label{thm5}\mbox{}

  \qitem{(a)}For every planar graph~$G$ we have $\chi(\edp{G}{3})\le105$.

  \smallskip
  \qitem{(b)}There exists a planar graph~$G_5$ such that
  $\chi(\edp{G_5}{3})=7$.
\end{theo}

\noindent
For outerplanar graphs~$G$ we have that $\chi(\edp{G}{3})\le10$, while
there exists an outerplanar graph~$G_4$ such that $\chi(\edp{G_4}{3})=5$
(see the results in Sections~\ref{sec3} and~\ref{sec4}).

\subsection{Generalised colouring numbers and main results}\label{ssec1.2}

When solving an optimisation problem it is often useful to preorder the
input so as to minimise some parameter. One such parameter is the
\emph{colouring number} $\col(G)$ of a graph $G$. This is the minimum
integer $k$ such that there is a linear ordering $L$ of $V$ such that every
vertex~$y$ has at most $k-1$ neighbours $x$ with $x<_Ly$. (So the colouring
number is one more than the \emph{degeneracy} of a graph.) It is well-known
and easy to see that the chromatic number $\chi(G)$ of a graph $G$
satisfies $\chi(G)\le\col(G)$. Although this bound is far from being tight
in many cases, it is often used to show that a specific class of graphs has
bounded chromatic number.

Different generalisations of the colouring number can be found in the
literature. Chen and Schelp~\cite{CS93} proved that the class of planar
graphs has linear Ramsey number by also controlling, for all vertices $v$,
the number of smaller vertices that can be reached by a path of length two,
whose middle vertex is larger than $v$. Various versions of their idea were
applied by Kierstead and Trotter~\cite{KT94}, Kierstead~\cite{K00}, and
Zhu~\cite{Z08} to problems concerning the game chromatic number of graphs
and gave rise to the 2-colouring number defined below. In their study of
oriented game chromatic number of graphs, Kierstead and Trotter~\cite{KT01}
considered paths of length four with different configurations of ``large''
internal vertices, which later motivated the notions of 4-colouring number
and weak 4-colouring number. Kierstead and Yang~\cite{KY03} bounded the
game colouring number in terms of the 2-colouring number, and Kierstead and
Kostochka~\cite{KK09} applied game colouring number to a (non-game) packing
problem.

All of these notions are encompassed in the concepts of the
\emph{$k$-colouring number} and the \emph{weak $k$-colouring number} of a
graph, both of which were first introduced by Kierstead and
Yang~\cite{KY03}.

Let $G=(V,E)$ be a graph, $L$ a linear ordering of~$V$, and~$k$ a positive
integer. We say that a vertex $x\in V$ is \emph{$k$-accessible} from
$y\in V$ if $x<_Ly$ and there exists an $xy$-path~$P$ of length at most~$k$
such that $y<_Lz$ for all internal vertices~$z$ of~$P$. Similarly, if all
internal vertices~$z$ of~$P$ satisfy the less restrictive condition that
$x<_Lz$, then we say that~$x$ is \emph{weakly \mbox{$k$-accessible}}
from~$y$. Let $R_{L,k}(y)$ be the set of vertices that are $k$-accessible
from~$y$, and $Q_{L,k}(y)$ the set of vertices that are weakly
$k$-accessible from~$y$. The \emph{$k$-colouring number $\col_k(G)$} and
\emph{weak $k$-colouring number $\wcol_k(G)$} of a graph~$G$ are defined as
follows:
\begin{align*}
  \col_k(G)&= 1+\min_L\max_{y\in V}|R_{L,k}(y)|,\\
  \wcol_k(G)& = 1+\min_L\max_{y\in V}|Q_{L,k}(y)|.
\end{align*}
If we allow paths of any length (but still have restrictions on the
position of the internal vertices), we get $R_{L,\infty}(y)$,
$Q_{L,\infty}(y)$, the \emph{$\infty$-colouring number $\col_\infty(G)$}
and the \emph{weak $\infty$-colouring number $\wcol_\infty(G)$}.

We now state the main results of this paper.

\begin{theo}\label{thm10}\mbox{}

  \qitem{(a)}For every odd positive integer~$p$ and every graph~$G$ we have
  $\chi(\edp{G}{p})\le\wcol_{2p-1}(G)$.

  \smallskip
  \qitem{(b)}For every even positive integer~$p$ and every graph~$G$ we
  have $\chi(\edp{G}{p})\le\wcol_{2p}(G)\cdot\Delta(G)$.
\end{theo}

\begin{theo}\label{thm11}\mbox{}\\*
  Let $p$ be an odd positive integer and $G$ a graph. Set $q=\wcol_p(G)$.

  \qitem{(a)}We have
  $\chi(\edp{G}{p})\le \bigl(\bigl\lfloor\frac12p\bigr\rfloor+2\bigr)^q$.

  \qitem{(b)}If $G$ has odd girth at least $p+1$, then
  $\chi(\epp{G}{p})\le \bigl(\bigl\lfloor\frac12p\bigr\rfloor+2\bigr)^q$.
\end{theo}

\noindent
An interesting aspect of generalised colouring numbers is that these
invariants can also be seen as gradations between the colouring number
$\col(G)$ and two important minor monotone invariants, namely the
\emph{tree-width $\mathrm{tw}(G)$} and the \emph{tree-depth
  $\mathrm{td}(G)$} (which is the minimum height of a depth-first search
tree for a supergraph of $G$, see~\cite{NOdM06}). More explicitly, for
every graph $G$ we have the following relations.

\begin{prop}\label{pro:coltwtd}\mbox{}

  \qitem{(a)}$\col(G)= \col_1(G)\le \col_2(G)\le \dots\le \col_\infty(G)=
  \mathrm{tw}(G)+1$;

  \smallskip
  \qitem{(b)}$\col(G)= \wcol_1(G)\le \wcol_2(G)\le \dots\le
  \wcol_\infty(G)= \mathrm{td}(G)$.
\end{prop}

\noindent
The equality $\col_\infty(G)=\mathrm{tw}(G)+1$ was first proved
in~\cite{Getal16}. The equality $\wcol_\infty(G)=\mathrm{td}(G)$ is
\cite[Lemma~6.5]{NOdM12}.

Relations between the two sets of numbers exist as well. Clearly,
$\col_1(G)=\wcol_1(G)$ and $\col_k(G)\le\wcol_k(G)$. For the converse,
Kierstead and Yang~\cite{KY03} proved that $\wcol_k(G)\le(\col_k(G))^k$.
Note that this means that if one of the generalised colouring numbers is
bounded for a class of graphs (for some~$k$), then so is the other one.

Shortly after Ne\v{s}et\v{r}il and Ossona de Mendez~\cite{NOdM08}
introduced the notion of \emph{classes with bounded expansion}, Zhu
provided, in~\cite{Z09}, a way of characterising these classes in terms of
the weak $k$-colouring numbers. We will use this characterisation as a
definition.

\begin{mydef}\label{def7}\mbox{}\\*
  {\rm A class of graphs~$\mathcal{K}$ has \emph{bounded expansion} if and
    only if there exist constants~$c_k$, $k=1,2,\ldots$ such that
    $\wcol_k(G)\le c_k$ for all~$k$ and all $G\in\mathcal{K}$.}
\end{mydef}

\noindent
By this definition, Theorem~\ref{thm3}\,(a) follows directly from both
Theorems~\ref{thm10}\,(a) and~\ref{thm11}\,(a).

We give the proofs of Theorems~\ref{thm10} and~\ref{thm11} in the next
section. The proof of Theorem~\ref{thm11} actually proves a stronger
result. For two graphs $G=(V,E)$ and $G'=(V,E')$ on the same vertex set,
define $G\cup G'=(V,E\cup E')$. Then the upper bound in both parts of
Theorem~\ref{thm11} holds for
$\chi(\edp{G}{1}\cup\edp{G}{3}\cup\cdots\cup\edp{G}{p})$ and
$\chi(\epp{G}{1}\cup\epp{G}{3}\cup\cdots\cup\epp{G}{p})$, respectively.

A natural question is if for even~$p$ we can generalise the bound in
Theorem~\ref{thm10}\,(b) by a similar bound
$\chi(\epp{G}{2}\cup\epp{G}{4}\cup\cdots\cup\epp{G}{p})\le
C\cdot\Delta(G)$, where $C$ depends on the generalised colouring numbers.
But this is not possible. Let~$T_{\Delta,2}$ be the $\Delta$-regular tree
of radius~2. Then we have $\wcol_1(T_{\Delta,2})=1$ and
$\wcol_k(T_{\Delta,2})=2$ for all $k\ge2$. It is easy to check that
$\chi(\epp{T_{\Delta,2}}{2})=\chi(\epp{T_{\Delta,2}}{4})=\Delta$, but
$\chi(\epp{T_{\Delta,2}}{2}\cup\epp{T_{\Delta,2}{4}})=\Delta(\Delta-1)+1$.
These examples generalise to larger distances.

The results in Theorem~\ref{thm11} are best possible in the sense that they
give upper bounds of $\chi(\edp{G}{p})$ and $\chi(\epp{G}{p})$ that depend
on $\wcol_p(G)$ only, whereas no such results are possible that depend on
$\wcol_k(G)$ with $k<p$. To see this, for $n,p\ge2$ let $S_{n,p}$ be the
$(p-1)$-subdivision of the complete graph~$K_n$ (that is, the graph formed
by replacing the edges of~$K_n$ by paths of length~$p$). Then we obviously
have $\chi(\edp{S_{n,p}}{p})=n$. On the other hand we have
$\wcol_{p-1}(S_{n,p})\le p+1$. To verify this, order the vertices of
$S_{n,p}$ as follows. First order the branch vertices (the vertices in the
original clique), and then order the subdivision vertices in any way.
Clearly, each branch vertex will not weakly $(p-1)$-access any other
vertex. An internal vertex of a subdivided edge can only weakly
$(p-1)$-access the other $p$ vertices on the path that replaced the edge
(including the two end-vertices of the path). So for fixed odd $p\ge3$ we
cannot bound $\chi(\edp{S_{n,p}}{p})$ by an expression that involves
$\wcol_{p-1}(S_{n,p})$ only.

The bound on the odd girth in Theorem~\ref{thm11}\,(b) is also best
possible. To show this, for $k,p\ge1$ let $A_{k,p}$ be formed by taking the
path $P_{p-1}$ of length $p-2$, and adding~$k$ new vertices that are
adjacent to both end-vertices of $P_{p-1}$ only. It is clear that if~$p$ is
odd, then $A_{k,p}$ has odd girth~$p$. Since between any of the~$k$ extra
vertices there is a path of length~$p$, we have
$\chi(\epp{A_{k,p}}{p})\ge k$. The ordering obtained by taking the two
end-vertices of $P_{p-1}$ first, and then ordering the other vertices in
any way, shows that $\wcol_p(A_{k,p})\le p-1$. So for fixed odd $p\ge3$ we
cannot bound $\chi(\epp{A_{k,p}}{p})$ by an expression that involves
$\wcol_p(A_{k,p})$ only.

Ne\v{s}et\v{r}il and Ossona de Mendez \cite[Section 11.9.3]{NOdM12} give
examples that even if we replace ``there exists a path of length~$p$
between~$x$ and~$y$'' by ``there exists an \emph{induced} path of
length~$p$ between~$x$ and~$y$'' in the definition of $\epp{G}{p}$, it is
not possible to reduce the bound on the odd girth in
Theorem~\ref{thm11}\,(b).

Finally, we point out a connection between the bound on
$\chi(\epp{G}{1}\cup\epp{G}{3}\cup\cdots\cup\epp{G}{p})$ in the proof of
Theorem~\ref{thm11}\,(b) and results of Naserasr et al.~\cite{NSS16}. For a
positive integer~$p$ and graph $G=(V,E)$, let the \emph{$p$-th walk
  power~$G^{(p)}$ of~$G$} be the graph with vertex set~$V$ and where~$xy$
is an edge if and only if there exists a walk of length~$p$ between~$x$
and~$y$. It is easy to see (see also Lemma~\ref{lem2}) that for odd~$p$,
if~$G$ has odd girth at least $p+1$, then for any two vertices $x,y\in
V(G)$ there exists a walk of length~$p$ between~$x$ and~$y$ if and only if
there exists a path of odd length at most~$p$ between~$x$ and~$y$. Hence
for odd~$p$, if~$G$ has odd girth at least $p+1$, then $G^{(p)}$ is
isomorphic to $\chi(\epp{G}{1}\cup\epp{G}{3}\cup\cdots\cup\epp{G}{p})$. So
it follows from \cite[Theorem~13]{NSS16} that for odd~$p$ there exist
planar graphs~$G$ with odd girth at least $p+1$ such that
$\chi(\epp{G}{1}\cup\epp{G}{3}\cup\cdots\cup\epp{G}{p})=\chi(G^{(p)})\ge
2^{p+1}$.

\subsection{Explicit upper bounds}\label{ssec3.0}

The upper bounds obtained by Ne\v{s}et\v{r}il and Ossona de Mendez in their
proof of Theorem~\ref{thm3}\,(a) are very large, even for $p=3$. Their
proof relies on the concept of \emph{$p$-centred colourings} of graphs. A
(proper) colouring of a graph~$G$ is a \emph{$p$-centred colouring} if for
each connected induced subgraph~$H$ of~$G$, either one colour appears
exactly once on~$H$ or~$H$ gets at least~$p$ colours. This is what is
proved in~\cite{NOdM12}.

\begin{theo}[Ne\v{s}et\v{r}il~\& Ossona de
  Mendez~\cite{NOdM12}]\label{realprev}\mbox{}\\*
  Let $p$ be an odd positive integer. If a graph $G$ has a $p$-centred
  colouring that uses at most $N=N(p)$ colours,
  then $\chi(\edp{G}{p})\le N2^{N2^N}$.
\end{theo}

\noindent
Given a graph~$G$, the \emph{star chromatic number} $\chi_s(G)$ is the
smallest number of colours needed to properly colour~$G$ such that every
two colours induce a star forest (a forest where every component is
isomorphic to a \emph{star $K_{1,m}$}). It is easy to see that a colouring
of a graph is 3-centred if and only if every two colours induce a star
forest. Albertson et al.~\cite{Aletal04} showed that the star chromatic
number of planar graphs is at most~20, and there exist planar graphs with
star chromatic number equal to~10. This means that the best upper bound
known for $\chi(\edp{G}{3})$ for planar graphs $G$ given by
Theorem~\ref{realprev} is $5\cdot2^{20,971,522}$, while the best possible
upper bound for planar graphs that can be found using that theorem directly
is $5\cdot2^{10,241}$.

An alternative bound can be obtained from Theorem~\ref{realprev} using the
following result.

\begin{theo}[Zhu~\cite{Z09}]\mbox{}\\*
  Every graph $G$ has a $p$-centred colouring that uses at most
  $\wcol_{2^{p-2}}(G)$ colours.
\end{theo}

\begin{cor}\label{first}\mbox{}\\*
  Let $p$ be an odd positive integer and~$G$ a graph. Setting
  $W=\wcol_{2^{p-2}}(G)$ we have $\chi(\edp{G}{p})\le W2^{W2^W}$.
\end{cor}

\noindent
More recently, Stavropoulos~\cite{S16} improved Corollary~\ref{first}.

\begin{theo}[Stavropoulos~\cite{S16}]\label{konstant}\mbox{}\\*
  For every odd integer $p\ge3$ and every graph $G$ we have
  $\chi(\edp{G}{p})\le \wcol_{2p-3}(G)2^{\wcol_{2p-3}(G)}$.
\end{theo}

\noindent
The best upper bound known for the weak colouring numbers of planar graphs
is given by the following result.

\begin{theo}[Van den Heuvel et
  al.~\cite{vdHetal16}]\label{binomcol}\mbox{}\\*
  For every positive integer~$k$ and planar graph~$G$ we have
  $\displaystyle\wcol_{k}(G)\le \binom{k+2}{2}\cdot(2k+1)$.
\end{theo}

\noindent
So $\wcol_2(G)\le30$ and $\wcol_3(G)\le70$ for planar $G$, which, when
combined with Corollary~\ref{first}, unfortunately gives a worse bound for
$\chi(\edp{G}{3})$ than the one using the star chromatic number obtained
earlier. Theorems~\ref{konstant} and~\ref{binomcol} together give
$\chi(\edp{G}{3})\le 70\cdot2^{70}$ for every planar graph~$G$, while
combining Theorems~\ref{thm10}\,(a) and~\ref{binomcol} already gives the
significantly better upper bound $\chi(\edp{G}{3})\le231$. In
Section~\ref{sec3} we will show that this bound can be lowered further to
105.

\medskip
The remainder of this paper is organised as follows. In the next section we
prove our main results, Theorems~\ref{thm10} and~\ref{thm11}. We use the
results from that section in Section~\ref{sec3} to find explicit upper
bounds for the chromatic number of exact distance graphs for some specific
classes of graphs, including graphs with bounded genus, graphs with bounded
tree-width, and graphs without a specified complete minor. In
Section~\ref{sec4} we describe the graph promised in
Theorem~\ref{thm5}\,(b). We close with a number of open problems and
directions for further study.

\section{Proofs of the main results}\label{proofs}

We need a few more definitions. For a positive integer $k$, we denote
$[k]=\{1,2,\dots,k\}$. For a vertex $v\in V$, we will denote by $N^k(y)$
the \emph{$k$-th neighbourhood} of $y$, that is, the set of vertices
different from~$v$ with distance at most~$k$ from~$v$; and we set
$N^k[v]=N^k(v)\cup\{v\}$. As is standard, we write~$N(v)$ for $N^1(v)$.

\subsection{Proof of Theorem~\ref{thm10}}\label{proof10}

For later use, we actually prove a slightly stronger result, which involves
a more technical variant of the generalised colouring numbers. Let
$G=(V,E)$ be a graph, $L$ a linear ordering of~$V$, and~$k$ a positive
integer. For a vertex $y\in V$, let $D_{L,k}(y)$ be the set of vertices~$x$
such that there is an $xy$-path $P_x=z_0,\dots,z_s$, with $x=z_0$, $y=z_s$,
of length $s\le k$, such that~$x$ is the minimum vertex in $P_x$ with
respect to~$L$, and such that $y\le_Lz_i$ for
$\bigl\lfloor\frac12k\bigr\rfloor+1\le i\le s$. We define the
\emph{distance-$k$-colouring number $\dcol_k(G)$} of a graph $G$ as
follows:
\[\dcol_k(G)=1+ \min_L\max_{y\in V}|D_{L,k}(y)|.\]
Since $R_{L,k}(y)\subseteq D_{L,k}(y)\subseteq Q_{L,k}(y)$ for every
ordering~$L$, distance~$k$ and vertex~$y$, we obtain
$\col_k(G)\le \dcol_k(G)\le \wcol_k(G)$. On the other hand, we also have
$Q_{L,\lfloor k/2\rfloor+1}(y)\subseteq D_{L,k}(y)$, which implies that
$\wcol_{\lfloor k/2\rfloor+1}(G)\le \dcol_k(G)$.

We will prove the following sharpening of Theorem~\ref{thm10}.

\begin{theo}\label{thm10d}\mbox{}

  \qitem{(a)}For every odd positive integer~$p$ and every graph~$G$ we have
  $\chi(\edp{G}{p})\le\dcol_{2p-1}(G)$.

  \smallskip
  \qitem{(b)}For every even positive integer~$p$ and every graph~$G$ we
  have $\chi(\edp{G}{p})\le\dcol_{2p}(G)\cdot\Delta(G)$.
\end{theo}

\begin{proof}
  (a)\quad For an odd positive integer~$p$ and graph $G=(V,E)$, set
  $q=\dcol_{2p-1}(G)$ and let~$L$ be an ordering of $V$ that witnesses
  $\max\limits_{y\in V}|D_{L,2p-1}(y)|=q-1$. Moving along the ordering~$L$
  we assign to each vertex $y\in V$ a colour $a(y)\in[q]$ that is different
  from~$a(x)$ for all $x\in D_{L,2p-1}(y)$. Next, define $\mu(y)$ as the
  minimum vertex with respect to~$L$ of the vertices in
  $N^{\lfloor p/2\rfloor}[y]$, and define $h:V\to[q]$ by $h(y)=a(\mu(y))$.
  We claim that $h$ is a $q$-colouring of~$\edp{G}{p}$.

  Consider an edge $uv\in E(\edp{G}{p})$. So there exists a path
  $P=z_0,z_1,\dots,z_p$ with $z_0=u$ and $z_p=v$. Clearly,
  $N^{\lfloor p/2\rfloor}[u]\cap N^{\lfloor p/2 \rfloor}[v]=\varnothing$,
  and hence $\mu(x)\ne\mu(y)$. Without loss of generality, assume
  $\mu(u)<_L\mu(v)$. Since
  $\mu(u),z_{\lfloor p/2\rfloor}\in N^{\lfloor p/2\rfloor}[u]$, there
  exists a path~$S_1$ between $\mu(u)$ and~$z_{\lfloor p/2\rfloor}$ of
  length at most $2\bigl\lfloor\frac12p\bigr\rfloor=p-1$ such that
  $V(S_1)\subseteq N^{\lfloor p/2\rfloor}[u]$. Similarly, there exists a
  path~$S_2$ between $z_{\lfloor p/2\rfloor+1}$ and $\mu(v)$ of length at
  most $p-1$ such that $V(S_2)\subseteq N^{\lfloor p/2\rfloor}[v]$. Since
  $N^{\lfloor p/2\rfloor}[u]\cap N^{\lfloor p/2 \rfloor}[v]=\varnothing$
  and $z_{\lfloor p/2\rfloor}z_{\lfloor p/2\rfloor+1}\in E$, we can combine
  these paths to a path~$S$ between $\mu(u)$ and $\mu(v)$ of length at most
  $2p-1$.

  Note that if we write $S=w_0,w_1,\ldots,w_t$ with $w_0=\mu(u)$ and
  $w_t=\mu(v)$, then the vertices~$w_i$ for
  $\bigl\lfloor\frac12k\bigr\rfloor+1\le i\le t$ all lie on~$S_2$, hence
  are in $N^{\lfloor p/2\rfloor}[v]$. Since $\mu(v)$ is the minimum vertex
  in $N^{\lfloor p/2\rfloor}[v]$, we have $\mu(v)\le_Lw_i$ for those~$w_i$.
  Thus~$S$ witnesses that $\mu(u)\in D_{L,2p-1}(\mu(v))$. We conclude that
  $h(u)=a(\mu(u))\ne a(\mu(v))=h(v)$, as required.

  \medskip\noindent
  (b)\quad For an even positive integer~$p$ and graph $G=(V,E)$, set
  $q=\dcol_{2p}(G)$ and let~$L$ be an ordering of $V$ that witnesses
  $\max\limits_{y\in V}|D_{L,2p}(y)|=q-1$. Moving along the ordering~$L$ we
  assign to each vertex $y\in V$ a colour $a(y)\in[q]$ that is different
  from~$a(x)$ for all $x\in D_{L,2p}(y)$. Additionally, for each
  vertex~$y$, choose an injective function $c_y:N(y)\to[\Delta(G)]$.

  Next, define $\mu(y)$ as the minimum vertex with respect to~$L$ of the
  vertices in $N^{p/2}[y]$. We also choose an arbitrary vertex in
  $N(\mu(y))\cap N^{p/2-1}(y)$; call it $\beta(y)$. To each vertex $y$ we
  assign as its colour the pair $(a(\mu(y)),c_{\mu(y)}(\beta(y))$. It is
  clear that this colouring uses at most $q\cdot \Delta(G)$ colours, and we
  claim that it is a proper colouring of $\edp{G}{p}$.

  Consider an edge $uv\in E(\edp{G}{p})$. First suppose that
  $\mu(u)\ne\mu(v)$. Then we can follow the proof of part~(a) to conclude
  that $a(\mu(u))\ne a(\mu(v))$, and hence the colours of~$u$ and~$v$
  differ in the first coordinate.

  So we are left with the case $\mu(u)=\mu(v)$. Since $d_G(u,v)=p$, we have
  that $\mu(v)\in N^{p/2}(u)\cap N^{p/2}(v)$, while
  $N^{p/2-1}(u)\cap N^{p/2-1}(v)=\varnothing$. This means that
  $\beta(u)\ne\beta(v)$. Together with the fact that
  $\beta(u),\beta(v)\in N(\mu(v))$, we obtain that
  $c_{\mu(v)}(\beta(u))\ne c_{\mu(v)}(\beta(v))$. This gives that the
  colours of~$u$ and~$v$ differ in the second coordinate, which completes
  the proof.
\end{proof}

\subsection{Proof of Theorem~\ref{thm11}}\label{proof11}

In the proof of Theorem~\ref{thm11} we use the following lemmas.

\begin{lem}\label{twoweak}\mbox{}\\*
  Let $G=(V,E)$ be a graph and~$L$ a linear ordering of~$V$. Let $x,y,z$ be
  distinct vertices in~$G$. If~$x$ is weakly $k$-accessible from~$y$,
  and~$z$ is weakly $\ell$-accessible from~$y$, then~$x$ is weakly
  $(k+\ell)$-accessible from~$z$ or~$z$ is weakly $(k+\ell)$-accessible
  from~$x$.
\end{lem}

\begin{proof}
  Since~$x$ is weakly $k$-accessible from~$y$, there is a path
  $x,v_1,v_2,\ldots,v_{r-1},y$ of length $r\le k$ for which all internal
  vertices~$v_i$ satisfy $x<_Lv_i$. Also, since~$z$ is weakly
  $\ell$-accessible from~$y$, there is a path $y,u_1,u_2,\ldots,u_{s-1},z$
  of length $s\le\ell$ for which all internal vertices~$u_j$ satisfy
  $z<_Lu_j$. Then, if $x<_Lz$, there is an $xz$-path of length at most
  $k+\ell$ with all internal vertices greater than~$x$ in~$L$; hence, $x$
  is weakly $(k+\ell)$-accessible from~$z$. Similarly, if $z<_Lx$, then~$z$
  is weakly $(k+\ell)$-accessible from~$x$.
\end{proof}

\begin{lem}\label{lem2}\mbox{}\\*
  Let $p$ be a positive integer and~$G$ a graph with odd girth at least
  $p+1$

  \smallskip
  \qitem{(a)} Every closed walk of odd length has length at least $p+1$.

  \smallskip
  \qitem{(b)} Let $x,y$ be different vertices and~$W$ a walk between~$x$
  and~$y$ of length $r\le p$. Then there exists a path between~$x$ and~$y$
  of length $s\le r$ such that~$s$ and~$r$ have the same parity.
\end{lem}

\begin{proof}
  The proof of (a) is straightforward, since a closed walk of odd length
  contains a cycle of odd length. For~(b), let $W=w_0,\ldots,w_r$, with
  $x=w_0$ and $y=w_r$. If $W$ itself is not a path, then some vertex~$z$
  appears more than once in~$W$. The part of~$W$ between the first and last
  appearances of~$z$ is a closed walk~$W'$ of length $t\le r$. Using~(a) we
  obtain that~$t$ must be even. Hence, if we remove~$W'$ from~$W$, we get a
  shorter walk between~$x$ and~$y$ of length $r-t\equiv r\pmod2$.
  Additionally, the resulting walk has fewer vertices that appear more than
  once than~$W$ does. Hence, if we do not immediately obtain a path, we can
  repeat this procedure inductively until we obtain an $xy$-path with the
  desired property.
\end{proof}

\begin{proof}[Proof of Theorem~\ref{thm11}]\mbox{}\\*
  For both parts of the theorem we use the same colouring. Let $L$ be an
  ordering of $V$ such that $\max\limits_{y \in V}|Q_{L,p}(y)|=q-1$. We
  first create an auxiliary colouring $a(y)\in[q]$ by moving along the
  ordering~$L$, and assigning to each vertex $y\in V$ a colour $a(y)\in[q]$
  that is different from~$a(x)$ for all $x\in Q_{L,p}(y)$. Next, for a
  vertex $x\in Q_{L,\lfloor p/2\rfloor}(y)$, let $d_y'(x)$ be the minimum
  integer~$k$ such that~$x$ is weakly \mbox{$k$-accessible} from~$y$, and
  set $d_y'(y)=0$.

  Define the function
  $b_y:[q]\rightarrow
  \bigl[\bigl\lfloor\frac12p\bigr\rfloor\bigr]\cup\{-1,0\}$ as follows. For
  a colour $c\in[q]$, let
  \[b_y(c)=\left\{
    \begin{array}{rl}
      d_y'(x),& \text{if there exists an
        $x\in Q_{L,\lfloor p/2\rfloor}(y)\cup\{y\}$ with
        $a(x)=c$};\\[1mm]
      -1,& \text{otherwise}.
    \end{array}\right.\]

  \noindent
  By Lemma~\ref{twoweak} and the definition of $a(x)$, we see that if
  $x\in Q_{L,\lfloor p/2\rfloor}(y)\cup\{y\}$ satisfies $a(x)=c$, then~$x$
  is the only vertex in $Q_{L,\lfloor p/2\rfloor}(y)\cup\{y\}$ with
  colour~$c$. That implies that $b_y$ is well defined.

  The number of possible functions
  $b_y:[q]\rightarrow
  \bigl[\bigl\lfloor\frac12p\bigr\rfloor\bigr]\cup\{-1,0\}$ is
  $\bigl(\bigl\lfloor\frac12p\bigr\rfloor+2\bigr)^q$. We will prove that
  labelling each vertex $y\in V$ with~$b_y$ gives a proper colouring for
  the graphs and situations described in parts~(a) and~(b) of the theorem.
  It is more convenient to do part~(b) first.

  \medskip
  \noindent
  (b)\quad Consider two vertices $u,v$ for which there exists a path of
  length~$p$ between~$u$ and~$v$. Without loss of generality we assume
  $u<_Lv$. If~$u$ is weakly $p$-accessible from~$v$ in~$L$, then we know
  that $a(u)\ne a(v)$, and hence $b_u(a(u))=0\ne b_v(a(u))$.

  So we are left with the case in which~$u$ is not weakly $p$-accessible
  from~$v$ in~$L$. Let~$k$ be the length of the shortest odd-length path
  between~$u$ and~$v$. We obviously have $k\le p$. Because~$u$ is not
  weakly $p$-accessible from~$v$ in~$L$, we also have $k\ne1$, hence
  $k\ge3$. Let $P=z_0,z_1,z_2,\ldots,z_{k-1},z_k$ be a path of length~$k$
  between $u=z_0$ an $v=z_k$. Let~$z_\ell$ be the vertex of~$P$ that is
  minimum with respect to the ordering~$L$. Since $u<_Lv$, we get that
  $z_\ell\ne v$, and, since~$u$ is not weakly $p$-accessible from~$v$, we
  see that $z_\ell\ne u$. Therefore,~$z_\ell$ is weakly $\ell$-accessible
  from~$u$ and weakly $(k-\ell)$-accessible from~$v$.

  First consider the case that $\ell<k-\ell$. Then $\ell<\frac12k$. We want
  to prove that $d'_u(z_\ell)=\ell$. For this, assume that
  $d'_u(z_\ell)=m<\ell$. Hence there is a path~$A$ between~$u$ and~$z_\ell$
  of length~$m$. If~$\ell$ and~$m$ have different parity, then the union
  of~$A$ and the path $z_0,z_1,\ldots,z_\ell$ gives a closed walk of odd
  length $m+\ell<2\ell<k\le p$, which contradicts Lemma~\ref{lem2}\,(a).
  So~$m$ and~$\ell$ have the same parity. Now if we replace in the path~$P$
  the part $z_0,z_1,\ldots,z_\ell$ with~$A$, we get a walk between~$u$
  and~$v$ of length $k-\ell+m<k$, hence with odd length. By
  Lemma~\ref{lem2}\,(b), this walk contains a path between~$u$ and~$v$ of
  odd length at most $k-\ell+m<k$, which contradicts the choice of~$P$.

  So we know that $d'_u(z_\ell)=\ell$. Notice that since there is a path of
  length $k-\ell$ between~$z_\ell$ and~$v$, we have that
  $d'_v(z_\ell)\le k-\ell\le p-\ell$. Since $\ell<\frac12k\le\frac12p$, we
  have that $z_\ell\in Q_{L,\lfloor p/2\rfloor}(u)$, and hence
  $b_u(a(z_\ell))=\ell$.

  Now consider a vertex $x\in Q_{L,\lfloor p/2\rfloor}(v)$ with
  $d_v'(x)=\ell$. We first prove that $x\ne z_\ell$. For suppose this is
  not the case, then there is a path from~$v$ to~$z_\ell$ of length~$\ell$.
  Together with the part of $z_\ell,z_{\ell+1},\ldots,z_k=v$ from the
  path~$P$, this gives a closed walk of length $k\le p$. Since~$k$ is odd,
  this contradicts Lemma~\ref{lem2}\,(a).

  Since $d'_v(x)=\ell$, $d_v'(z_\ell)\le p-\ell$ and $x\ne z_\ell$, by
  Lemma~\ref{twoweak} we get that~$x$ is weakly $p$-accessible
  from~$z_\ell$ or~$z_\ell$ is weakly $p$-accessible from~$x$. This gives
  $a(x)\ne a(z_\ell)$, which implies, by choice of~$x$, that
  $b_v(a(z_\ell))\ne\ell$.

  If $k-\ell<\ell$, we can prove in a similar way that $b_u\ne b_v$, which
  completes the proof of part~(b) of the theorem.

  \medskip
  \noindent
  (a)\quad This time we consider two vertices $u,v$ that have distance~$k$
  in~$G$, for some odd integer $k\le p$. (To prove the statement, it would
  be enough to prove the case $k=p$, but we prefer to give the proof of a
  more general statement.) We can more or less follow the proof of part~(b)
  above, working with a shortest path $P=z_0,z_1,z_2,\ldots,z_{k-1},z_k$
  between $u=z_0$ and $v=z_k$.

  Since~$P$ is a shortest path, we immediately get that
  $d_u'(z_\ell)=d_G(u,z_\ell)=\ell$ and
  $d_v'(z_\ell)=d_G(v,z_\ell)=p-\ell$. This also means that $x\ne z_\ell$,
  since $d_G(v,x)\le d_v'(x)=\ell<p-\ell$. For the remainder, the proofs
  are exactly the same.
\end{proof}

\noindent
The proofs of Theorem~\ref{thm11}\,(a) and~(b) above give results that are
stronger than the statements in the theorem. We already discussed in
Subsection~\ref{ssec1.2} that in fact we prove upper bounds on
$\chi(\edp{G}{1}\cup\edp{G}{3}\cup\cdots\cup\edp{G}{p})$ and
$\chi(\epp{G}{1}\cup\epp{G}{3}\cup\cdots\cup\epp{G}{p})$. Additionally, in
part~(a) we could replace the condition that we add an edge $uv$ to
$\edp{G}{p}$ if $d_G(u,v)=p$, i.e.\ ``there is a shortest path of
length~$p$ between~$u$ and~$v$'', by the weaker condition ``there is a
path~$P$ of length~$p$ between~$u$ and~$v$ such that any shorter path
between those vertices is internally disjoint from~$P$''.

\section{Explicit upper bounds on the chromatic number of exact distance
  graphs}\label{sec3}

In this section we use Theorem~\ref{thm10d}\,(a) to find explicit upper
bounds for the chromatic number of exact distance graphs for certain types
of graphs, including planar graphs, graphs with bounded tree-width, and
graphs without a complete minor. Obtaining these bounds involves finding
upper bounds for the distance-$k$-colouring numbers $\dcol_k(G)$. More
explicitly, we will prove the following results.

\begin{theo}\label{dcolplan}\mbox{}\\*
  Let $k$ be a positive integer.

  \qitem{(a)}For every planar graph $G$ we have
  $\displaystyle\dcol_k(G)\le
  \binom{\lfloor k/2\rfloor+3}{2} \cdot(2k+1)-k$.

  \smallskip
  \qitem{(b)}For every graph $G$ with genus $g$ we have
  $\displaystyle\dcol_k(G)\le
  \Bigl(2g+\binom{\lfloor k/2\rfloor+3}{2}\Bigr)\cdot(2k+1)-k$.
\end{theo}

\begin{theo}\label{dcoltw}\mbox{}\\*
  Let $k$ and $t$ be positive integers. For every graph $G$ with tree-width
  at most $t$ we have
  $\displaystyle\dcol_{k}(G)\le
  \binom{t+\lfloor k/2\rfloor+1}{t}$.
\end{theo}

\begin{theo}\label{minor}\mbox{}\\*
  Let $k$ and $t$ be positive integers with $t\ge4$. For every $K_t$-minor
  free graph $G$ we have
  $\displaystyle\dcol_k(G)\le
  \binom{t+\lfloor k/2\rfloor-1}{t-2}\cdot(t-3)(2k+1)$.
\end{theo}

\noindent
Since outerplanar graphs~$G$ have tree-width at most~2, combining
Theorems~\ref{thm10d}\,(a) and~\ref{dcoltw} gives $\chi(\edp{G}{3})\le10$.
Similarly, from Theorem~\ref{dcolplan} we see that for planar graphs~$G$ we
have $\chi(\edp{G}{3})\le105$, while for graphs~$G$ embeddable on the torus
we have $\chi(\edp{G}{3})\le127$.

We will prove those theorems in the remainder of this section. They are
based on the methods developed in Van den Heuvel et al.~\cite{vdHetal16} to
obtain bounds for the generalised colouring numbers.

\subsection{Graphs with bounded tree-width}\label{ssec3.1}

Recall that Proposition~\ref{pro:coltwtd} tells us that
$\col_\infty(G)=\mathrm{tw}(G)+1$. In~\cite{Getal16}, Grohe et al.\
provided a sharp upper bound for the weak colouring numbers $\wcol_k(G)$ of
a graph $G$ in terms of its tree-width. The following result is implicit in
the proof of \cite[Theorem~4.2]{Getal16}.

\begin{lem}[Grohe et al.~\cite{Getal16}]\label{tw}\mbox{}\\*
  Let $G$ be a graph and $L$ a linear ordering of $V(G)$ with
  $\max\limits_{y\in V(G)}|R_{L,\infty}(y)|\le t$. For every positive
  integer $k$ and vertex $y\in V(G)$ we have
  $\displaystyle|Q_{L,k}(y)|\le \binom{t+k}{t}-1$.
\end{lem}

\noindent
Although we can define tree-width of a graph in terms of its
$\infty$-colouring number, in order to prove Theorem~\ref{dcoltw} we shall
make use of a better known definition which is in terms of \emph{k-trees}.
A \emph{k-tree} is a graph which is either a clique of size $k+1$ or is
obtained from a smaller $k$-tree by adding a vertex adjacent to $k$
vertices which are pairwise adjacent. The \emph{tree-width} of a graph $G$
is the smallest $k$ such that $G$ is a subgraph of a $k$-tree.

Let $G$ be a $k$-tree. For a fixed way of constructing $G$ from a
$(k+1)$-clique $K_0$ we obtain a linear ordering $L$ of $V(G)$ as follows.
Let the vertices of $K_0$ be the smallest in the ordering, and order them
in an arbitrary way. Then for $y\notin K_0$ we let $x<_L y$ if $x$ was
added to the $k$-tree before $y$. We call this a \emph{simplicial
  ordering}. For $y\notin K_0$ we note that, by definition of~$L$,
$R_{L,1}(y)$ induces a $k$-clique.

\begin{proof}[Proof of Theorem~\ref{dcoltw}.]
  Since $\dcol_k(G)$ cannot decrease if we add edges, we may assume
  that~$G$ is a $k$-tree. Let $L$ be a simplicial ordering derived as
  above, where we started with some $(k+1)$-clique~$K_0$ in~$G$. While in
  general we have $Q_{L,\lfloor k/2\rfloor+1}(y)\subseteq D_{L,k}(y)$, we
  shall prove that our choice of $G$ and $L$ implies
  $Q_{L,\lfloor k/2\rfloor+1}(y)= D_{L,k}(y)$, for every $k\ge1$ and
  $y\in V(G)$.

  Our first step in this direction will be proving that every vertex
  $y\in V(G)$ satisfies $R_{L,1}(y) =R_{L,\infty}(y)$. Notice that if
  $y\in V(G)$ belongs to~$K_0$, then $R_{L,\infty}(y)$ only contains
  vertices in~$K_0$ and, since $K_0$ induces a clique in $G$, all of these
  vertices belong to $R_{L,1}(y)$. So consider some $y\notin K_0$. From the
  construction of a $k$-tree, it follows that removing $R_{L,1(y)}$
  disconnects the graph, and that the component $C_y$ containing $y$
  satisfies $y<_Lz$ for all $ z \in C_y$, $z\ne y$. This tells us that any
  $xy$-path with $x<_Ly$ and $y<_L z$ for all internal vertices $z$ must
  have its interior in $C_y$. In turn, this implies that for such a path to
  exist we must have $x\in R_{L,1}(y)$. This shows
  $R_{L,1}(y)=R_{L,\infty}(y)$.

  Suppose $x,y\in V(G)$ satisfy $x\in D_{L,k}(y)$ for some integer $k\ge1$.
  By the definition of $D_{L,k}(y)$, we have that there is an $xy$-path
  $P=z_0,\dots,z_s$, with $x=z_0$, $y=z_s$, of length $s\le k$, such
  that~$x$ is the minimum vertex in $P$ with respect to~$L$, and such that
  $y\le_Lz_i$ for $\bigl\lfloor\frac12k\bigr\rfloor+1\le i\le s$. Let
  $0\le d\le s$ be the largest index such that $z_d<y$. The subpath
  $z_d,\dots,z_s$ of $P$ guarantees that that $z_d\in R_{L,\infty}(y)$.
  Since $R_{L,1}(y)=R_{L,\infty}(y)$, we know that $z_d\in N(y)$. By the
  definition of $P$ and choice of $d$, we also know that $d\le
  \bigl\lfloor\frac12k\bigr\rfloor$. Therefore, the path
  $z_0,\dots,z_d,z_s$ is an $xy$-path of length at most
  $\bigl\lfloor\frac12k\bigr\rfloor+1$ with no other restriction than the
  one that~$x$ is its minimum vertex with respect to~$L$. This means that
  $x\in Q_{L,\lfloor k/2\rfloor+1}(y)$. Since the choice of $x,y$ and~$k$
  was arbitrary, we have that $Q_{L,\lfloor k/2\rfloor+1}(y)= D_{L,k}(y)$
  for every integer $k\ge1$ and every $y\in V(G)$.

  Since our ordering satisfies $t\ge R_{L,1}(y)=R_{L,\infty}(y)$, the bound
  on $D_{L,k}(y)$ follows from Lemma~\ref{tw}.
\end{proof}

\noindent
It is possible to modify the examples in Grohe et al.~\cite{Getal16} to
show that the upper bounds on $\dcol_k(G)$ in Theorem~\ref{dcoltw} for
graphs with tree-width at most~$t$ are best possible.

\subsection{Graphs with excluded complete minors}\label{ssec3.2}

In order to provide upper bounds for the generalised colouring numbers for
graphs that exclude a fixed minor, Van den Heuvel et al.~\cite{vdHetal16}
constructed ordered vertex partitions where each part has neighbours in
only a bounded number of earlier parts and the intersection of each part
with the $k$-neighbourhood of an earlier part is also bounded. We will make
use of these decompositions for our proofs as well.

A \emph{decomposition} of a graph~$G$ is a sequence
$\mathcal{H}=(H_1,\ldots,H_\ell)$ of non-empty subgraphs of $G$ such that
the vertex sets $V(H_1),\ldots,V(H_\ell)$ partition $V(G)$. The
decomposition $\mathcal H$ is \emph{connected} if each $H_i$ is connected.

Let $\mathcal{H}=(H_1,\ldots,H_\ell)$ be a decomposition of a graph $G$,
$i$ a positive integer, and $C$ a component of
$G-\bigcup_{1\le j\le i}V(H_j)$. We define the \emph{$i$-th separating
  number of $C$} as $s_i(C)=|\{j\in[i]\mid E(C,H_j)\ne\varnothing\}|$,
where $E(C,H_j)$ is the set of edges with one end-vertex in $C$ and the
other end-vertex in $H_j$. Let $w_i(\mathcal{H})=\max s_i(C)$, where the
maximum is taken over all components~$C$ of $G-\bigcup_{1\le j< i}V(H_j)$.
We define the \emph{width} of $\mathcal{H}$ as
$W(\mathcal{H})=\max\limits_{1\le i\le\ell}w_i(\mathcal{H})$.

Let $G$ be a graph, let $H\subseteq G$ be a \emph{connected} subgraph of
$G$, and let $f:\mathbb{N} \to\mathbb{N}$ be a function. We say that $H$
\emph{$f$-spreads on $G$} if, for every $k\in\mathbb{N}$ and $v\in V(G)$,
we have
\[|N^k[v]\cap V(H)|\le f(k).\]
We say a decomposition $\mathcal{H}$ is \emph{$f$-flat} if each $H_i$
$f$-spreads on $G-\bigcup_{1\le j<i}V(H_j)$. We say $\mathcal{H}$ is a
\emph{flat decomposition} if $\mathcal{H}$ is an $f$-flat decomposition for
some function $f:\mathbb{N}\to\mathbb{N}$.

Van den Heuvel et al.~\cite{vdHetal16} related the width of a connected
decomposition to the tree-width of the minor obtained by contracting each
part.

\begin{lem}[Van den Heuvel et al.~\cite{vdHetal16}]\label{width}\mbox{}\\*
  Let $G$ be a graph, and let $\mathcal{H}=(H_1,\dots,H_\ell)$ be a
  connected decomposition of $G$ of width at most~$t$. By contracting each
  (connected) subgraph $H_i$ to a single vertex, we obtain a graph $H$ with
  $\ell$ vertices and tree-width at most $t$.
\end{lem}

\noindent
The proof of the lemma in~\cite{vdHetal16} shows the power of generalised
colouring numbers. It actually gives a short argument that the contracted
graph~$H$ satisfies $\col_\infty(H)\le t+1$. The bound on the tree-width
then follows by Proposition~\ref{pro:coltwtd}. Moreover, the proof shows
that the ordering~$L$ of $V(H)$ obtained by setting $H_i<_LH_j$ if $i<j$
satisfies $\max\limits_{1\le i\le\ell}|R_{L,\infty}(H_i)|\le t$. Using this
property we can prove that if the decomposition from which $H$ was obtained
is $f$-flat, then we can find an upper bound on $\dcol_k(G)$ in terms of
$f(k)$.

\begin{lem}\label{flatbound}\mbox{}\\*
  Let $f:\mathbb{N}\to\mathbb{N}$ and let $t,k$ be positive integers. For
  every graph $G$ that admits a connected $f$-flat decomposition of width
  at most~$t$ we have
  $\displaystyle\dcol_{k}(G)\le
  \binom{t+\lfloor k/2\rfloor+1}{t}\cdot f(k)$.
\end{lem}

\begin{proof}
  The proof of this lemma is similar to that of
  \cite[Lemma~3.5]{vdHetal16}. Let $\mathcal{H}=(H_1,\dots,H_\ell)$ be a
  connected $f$-flat decomposition of $G$ of width $t$. Since $\mathcal{H}$
  is connected, we know, by Lemma~\ref{width}, that contracting the
  subgraphs in~$\mathcal{H}$ leads to a graph $H$ with tree-width at
  most~$t$. We identify the vertices of~$H$ with the subgraphs $H_i$, and
  define a linear ordering $L$ on $V(H)$ by setting $H_i<_LH_j$ if $i<j$.
  By the proof of \cite[Lemma~3.1]{vdHetal16} we get that~$L$ satisfies
  $\max\limits_{1\le i\le\ell}|R_{L,\infty}(H_i)|\le t$. Using
  Lemma~\ref{tw} this implies that
  $|Q_{L,\lfloor k/2\rfloor +1}(H_i)|\le
  \binom{t+\lfloor k/2\rfloor+1}{t}-1$ for any vertex $H_i\in V(H)$.
  Arguing as in the proof of Theorem~\ref{dcoltw}, we see that for every
  $H_i\in V(H)$ we have
  $|D_{L,k}(H_i)|\le \binom{t+\lfloor k/2\rfloor+1}{t}-1$.

  From $L$ we define an ordering $L'$ on $V(G)$ in the following way. For
  $u\in H_i$ and $v\in H_j$ with $i\ne j$, we let $u<_{L'}v$ if $i<j$.
  Then, for every $1\le i\le \ell$, we order the vertices of $H_i$ in any
  order. It is easy to see that any vertex $v\in H_i$ satisfies
  \[D_{L',k}(v)\subseteq
  N^k[v]\cap\bigl(H_i\cup \{\,H_j\mid H_j\in D_{L,k}(H_i)\,\}\bigr).\]
  Hence, we have that there are at most
  $\binom{t+\lfloor k/2\rfloor+1}{t}$ subgraphs among $H_1,\ldots,H_{\ell}$
  in~$G$ that contain vertices from $D_{L',k}(v)$. Since $\mathcal{H}$ is
  $f$-flat, we know that the intersection of each of these subgraphs with
  $N^k[v]$ is at most $f(k)$. Finally, since $D_{L',k}(v)$ is a proper
  subset of $N^k[v]$ (as $v\notin D_{L',k}(v)$), the result follows.
\end{proof}

\noindent
Also in~\cite{vdHetal16}, it was proved that graphs that do not contain a
complete graph as a minor have flat decompositions of small width.

\begin{lem}[Van den Heuvel et
  al.~\cite{vdHetal16}]\label{flatslim}\mbox{}\\*
  Let $t\ge 4$ and let $f:\mathbb{N}\to\mathbb{N}$ be the function
  $f(k)=(t-3)(2k+1)$. For every $K_t$-minor free graph~$G$ we have that
  there is a connected $f$-flat decomposition of $G$ with width at most
  $t-2$.
\end{lem}

\noindent
Combining Lemmas~\ref{flatbound} and~\ref{flatslim} immediately gives
Theorem~\ref{minor}.

We say a path is \emph{optimal} if it is a shortest path between its
end-points. The following easy result states that a decomposition
$\mathcal{H}=(H_1,\dots,H_\ell)$ in which each subgraph $H_i$ is an optimal
path in $G-\bigcup_{1\le j< i}V(H_j)$ is $f$-flat for $f(k)=2k+1$. We call
such a decomposition an \emph{optimal-path decomposition}.

\begin{lem}[Van den Heuvel et al.~\cite{vdHetal16}]\label{short}\mbox{}\\*
  Let $G$ be a graph, $y$ be a vertex of~$G$, and $P$ be an optimal path in
  $G$. Then $P$ contains at most $2k+1$ vertices of the closed
  $k$-neighbourhood $N^k[y]$ of $y$.
\end{lem}

\noindent
Optimal-path decompositions of small width were found in~\cite{vdHetal16}
for planar graphs.

\begin{lem}[Van den Heuvel et al.~\cite{vdHetal16}]\label{maxp}\mbox{}\\*
  Every maximal planar graph has an optimal-path decomposition of width at
  most 2.
\end{lem}

\noindent
This lemma allows us, through Lemma~\ref{flatbound}, to prove
Theorem~\ref{dcolplan}.

\begin{proof}[Proof of Theorem~\ref{dcolplan}]
  We begin by proving part (a). Since $\dcol_k(G)$ cannot decrease when
  edges are added, we may assume that $G$ is maximal planar. By
  Lemma~\ref{maxp}, there exists an optimal-path decomposition
  $\mathcal{H}=(H_1,\dots,H_\ell)$ of~$G$ of width at most~2. The proof of
  Lemma~\ref{flatbound} tells us that since $G$ admits a connected
  decomposition of width at most $2$, there is an ordering $L'$ of $V(G)$
  such that at most $\binom{\lfloor k/2\rfloor+3}{2}$ subgraphs among
  $H_1,\ldots,H_{\ell}$ contain vertices from $D_{L',k}(v)$, for every
  integer \mbox{$k\ge1$} and $v\in V(G)$. This ordering is obtained from an
  ordering $L$ of the subgraphs $H_1,\dots,H_\ell$, where vertices in the
  same subgraph are ordered in an arbitrary way. This time we have that
  each subgraph~$H_i$ is an optimal path. We order each~$H_i$ in its path
  order. Take $y\in V(G)$. Then $y\in V(H_i)$ for some $1\le i\le\ell$.
  Lemma~\ref{short} tells us that an optimal-path decomposition is
  $(2k+1)$-flat. Therefore,~$D_{L',k}(y)$ contains at most $2k+1$ vertices
  of each of the at most $\binom{\lfloor k/2\rfloor+3}{2}-1$ subgraphs,
  other than $H_i$, which intersect $D_{L',k}(y)$. Meanwhile, $D_{L',k}(y)$
  contains at most $k$ vertices of~$H_i$, since we have ordered the optimal
  path~$H_i$ in its path order. We find that every vertex~$y$ in $G$
  satisfies
  \[|D_{L',k}(y)|\le (\binom{\lfloor k/2\rfloor+3}{2}-1)\cdot(2k+1)+k=
  \binom{\lfloor k/2\rfloor+3}{2}\cdot(2k+1)-k-1,\]
  which concludes the proof of part (a).

  The proof of part (b) is similar to the proof of
  \cite[Theorem~1.5\,(a)]{vdHetal16}. We assume $g>0$, as otherwise the
  result holds by Theorem~\ref{dcolplan}\,(a). It is well known (see e.g.\
  \cite[page~111]{MT01} and the proof of \cite[Theorem~1]{Q85}) that a
  graph of genus $g>0$ contains a non-separating cycle $C$ that consists of
  two optimal paths and such that $G-C$ has genus $g-1$. We construct a
  linear order $L$ of~$V(G)$ in the following way. The first vertices in
  $L$ will be the vertices in such a cycle $C$. If after removing that
  cycle the genus of the resulting graph is greater than~0, then we choose
  another such cycle, make its vertices the next ones in the ordering, and
  remove the cycle. We repeat this process inductively until the resulting
  graph is a planar graph $G'$. The vertices of~$G'$ are placed at the end
  of~$L$, ordered in the way that gives the bound on $\dcol_k(G')$ from
  Theorem~\ref{dcolplan}\,(a).

  Lemma~\ref{short} tells us that for any vertex $y$ and optimal path $P$
  we have $|N^k[y]\cap V(P)|\le 2k+1$ for every $k$. Hence
  $|D_{L,k}(y)\cap V(P)|\le 2k+1$ for every vertex $y$ and optimal
  path~$P$. It follows that for any vertex $y$ in $G$, the set $D_{L,k}(y)$
  can have at most $2g(2k+1)$ vertices on the removed cycles. (Each of the
  two optimal paths that form a cycle is optimal after the earlier cycles
  are removed, and vertices cannot belong to $D_{L,k}(y)$ through vertices
  in older cycles.) Only a vertex $x$ in the planar graph $G'$ can have
  other vertices of $G'$ in $D_{L,k}(x)$ and Theorem~\ref{dcolplan}\,(a)
  gives us a bound on the number of such vertices. Hence, we obtain that
  every vertex~$y$ in $G$ satisfies
  \[|D_{L,k}(y)|\le 2g\cdot(2k+1)+
  \binom{\lfloor k/2\rfloor+3}{2}\cdot(2k+1)-k-1.\]
  The result follows.
\end{proof}

\section{A lower bound on the chromatic number of exact distance-3 graphs
  of planar graphs}\label{sec4}

\begin{figure}[h]
 \centering
 \captionsetup{justification=centering}
 \includegraphics[height=2 in]{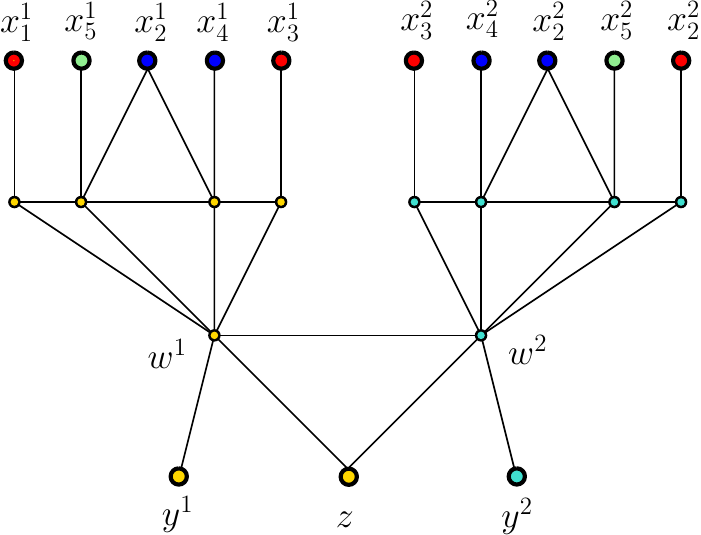}
 \smallskip
 \caption{ An outerplanar graph $G_4$ with $\chi(\edp{G_4}{3})=5$.}
  \label{five}
\end{figure}

In \cite[Exercise~11.4]{NOdM12} a planar graph~$G$ such that
$\chi(\edp{G}{3})=6$ is given (see also~\cite{NOdM15}). As we will prove
below, the outerplanar graph~$G_4$ in Figure~\ref{five} satisfies
$\chi(\edp{G_4}{3})=5$. We will use that graph to construct a planar graph
$G_5$ such that $\chi(\edp{G_5}{3})=7$.

\begin{mtheorem}[Theorem \ref{thm5}\,(b)]\mbox{}\\*
  There is a planar graph~$G_5$ such that $\chi(\edp{G_5}{3})=7$.
\end{mtheorem}

\begin{proof}
  We will prove first that $\chi(\edp{G_4}{3})=5$, using the vertex
  labelling provided in Figure~\ref{five}. Consider a proper colouring of
  $\edp{G_4}{3}$. Note that $C^1=x_1^1,x_2^1,x_3^1,x_4^1,x_5^1,x_1^1$ and
  $C^2=x_1^2,x_2^2,x_3^2,x_4^2,x_5^2,x_1^2$ form disjoint 5-cycles
  $\edp{G_4}{3}$. Hence, the vertices in $V(C^1)\cup V(C^2)$ need at
  least~3 colours. Given that $V(C^1)\cup V(C^2) \subseteq N(z)$ in
  $\edp{G_4}{3}$, if we use more than~3 colours on $V(C^1)\cup V(C^2)$,
  then we already use at least 5 colours. So assume that the vertices in
  $V(C^1)\cup V(C^2)$ are coloured with~3 colours only. Since
  $V(C^i)\subseteq N(y^i)$ in $\edp{G_4}{3}$ for $i=1,2$, and
  $y^1y^2\in E(\edp{G_4}{3})$, we need at least~2 extra colours. So we
  always use at least~5 colours in a proper colouring of~$\edp{G_4}{3}$.
  Figure~\ref{five} gives a colouring of~$G_4$ with~5 colours which is a
  proper colouring of $\edp{G_4}{3}$. This shows that
  $\chi(\edp{G_4}{3})=5$.

  Now let $F_1$ and $F_2$ be two disjoint copies of $G_4$. Let~$H$ be a
  path on 5 vertices, disjoint from~$F_1$ and~$F_2$, with vertices
  $y_1',w_1',z',w_2',y_2'$ in that order, together with the edge
  $w_1'w_2'$. (This is exactly the graph formed by the vertices
  $\{y^1,w^1,z,w^2,y^2\}$ in Figure~\ref{five}.) The graph~$G^-_5$ has
  vertex set and edge set:
  \begin{align*}
    V(G^-_5)&= V(F_1)\cup V(F_2)\cup V(H);\\
    E(G^-_5)&= E(F_1)\cup E(F_2)\cup E(H)\cup
    \{\,b_1w_1'\mid b_1\in V(F_1)\,\}\cup
    \{\,b_2w_2'\mid b_2\in V(F_2)\,\}.
  \end{align*}
  Finally, the graph~$G_5$ is obtained from~$G^-_5$ by subdividing once all
  the edges of the form $b_1w_1'$ and~$b_2w_2'$ (replacing each edge by a
  path of length 2). Since $G_4$ is outerplanar, it is easy to check
  that~$G_5$ is planar.

  If $u,v\in V(F_{1})$ and $P$ is a $uv$-path in $G_5$ but
  $V(P)\nsubseteq V(F_1)$, then $w_1'\in V(P)$. Thus the length of $P$ is
  at least 4. We conclude that if two vertices $u,v$ have distance 3 in
  $G_5$, then any shortest $uv$-path has all its vertices in $V(F_1)$.
  Therefore, the number of colours needed to colour the vertices of~$F_1$
  in $\edp{G_5}{3}$ is~5, and the same applies to~$F_2$. We now can argue
  as in the proof of $\chi(\edp{G_4}{3})=5$ above to reach the conclusion
  $\chi(\edp{G_5}{3})=7$.
\end{proof}

\noindent
Since the graph $G_4$ in Figure~\ref{five} is outerplanar, it does not
have~$K_4$ as a minor. Also, the graph~$G_5$ we constructed above is
planar, so does not have~$K_5$ as a minor. We can iterate the construction
to obtain graphs~$G_t$ that are $K_t$-minor free, for $t\ge4$, and for
which $\chi(\edp{G_t}{3})\ge2(t-2)+1$. To obtain $G_{t+1}$ from~$G_t$, we
take two copies of~$G_t$, one copy of the graph~$H$ from above, and add
paths of length 2 between all vertices in the first copy of~$G_t$
and~$w_1'$, and between all vertices in the second copy of~$G_t$
and~$w_2'$. It is straightforward to check that if~$G_t$ is $K_t$-minor
free, then $G_{t+1}$ is $K_{t+1}$-minor free, and that $\edp{G_{t+1}}{3}$
needs at least 2 more colours than $\edp{G_t}{3}$ does.

The property that for $t\ge5$ there exists a graph~$G$ that is $K_t$-minor
free and satisfies $\chi(\edp{G}{3})\ge2(t-2)+1$ does not extend to $t=3$.
To see this, note that the only graphs that are $K_3$-minor free are
acyclic graphs (i.e.\ forests), which implies they are bipartite. And for
bipartite graphs~$G$ we have that $\edp{G}{3}$ is bipartite as well (in
fact, even the exact $p$-power graph $\epp{G}{p}$ is bipartite for every
odd~$p$), hence $\chi(\edp{G}{3})\le2$.

Notice that one can construct the graph $G_4$ of Figure~\ref{five} (and the
graphs $G_t$ for $t\ge4$) by using operations similar to those of used in
the Haj\'os construction~\cite{H61}. Consider the graph~$S$ induced on
$G_4$ by $(N(w^1)\setminus w^2)\cup\{w_1,x_1^1,x_2^2,\ldots,x_5^1\}$. The
main connected component of the graph $\edp{S}{3}$ consists of a cycle and
two apex vertices, $z$ and $y_1$, that are adjacent to all the vertices in
the cycle. One can obtain $G_4$ by taking two copies of $S$, identifying
the two vertices that correspond to $z$, and adding an edge between the two
vertices that correspond to $w_1$. In the exact distance-3 graph, we see
that one of the apex vertices has been identified, while those that
correspond to $y_1$ have been joined by an edge. However, the operation of
deletion, used in the Haj\'{o}s construction, is not used in our
construction. This is mainly because we want to obtain a graph with
chromatic number strictly larger than that of the parts it is formed of.

\section{Discussion and open problems}

In this paper we give bounds on the chromatic number of exact distance
graphs for some classes of graphs. In general, the difference between the
best lower and upper bounds is still quite large, so we can't really claim
we have an insight of what the correct best possible bounds are.

When considering odd distances, one, trivial, example for which there are
tight bounds is the class of bipartite graphs. We noted at the end of
Section~\ref{sec4} that every bipartite graph $G$ satisfies
$\chi(\edp{G}{p})\le \chi(\epp{G}{p})\le 2$ for every odd~$p$.

Since our upper bounds are expressed in terms of generalised colouring
numbers they increase with the distance. In contrast, we do not provide
lower bounds which increase with the distance. Because of the difficulty in
providing lower bounds which depend on the distance, the following
question, attributed to Van den Heuvel and Naserasr, was asked in
\cite[Section~11.9.3]{NOdM12} (see also~\cite{NOdM15}): ``Is there a
constant~$C$ such that for every odd integer~$p$ and every planar graph~$G$
we have $\chi(G^{[\natural p]})\le C$\,?'' Very recently, Bousquet et
al.~\cite{Betal17} gave a negative answer to this question by constructing
a sequence of outerplanar graphs $U_3,U_5,\dots$ such that for every odd
$p\ge3$ we have
$\chi(U_p^{[\natural p]})\in \Omega\bigl(\frac{p}{\log p}\bigr)$. In
Section~\ref{sec3} we proved that if~$G$ has tree-width at most $t$ then
$\chi(G^{[\natural p]})\in \mathcal{O}(p^{t-1})$. This means that graphs
$G$ of tree-width at most $2$ satisfy
$\chi(G^{[\natural p]})\in \mathcal{O}(p)$. Therefore, for graphs of
tree-width at most $2$ (which includes outerplanar graphs), our upper
bounds are close to having the right order in terms of the distance.

\medskip
As we mentioned in Section~\ref{sec1}, the proof of Theorem~\ref{thm11}
actually gives that for a class of graphs $\mathcal{K}$ with bounded
expansion we can find a constant $N=N(\mathcal{K},p)$ such that
$\chi(\edp{G}{1}\cup\edp{G}{3}\cup\cdots\cup\edp{G}{p})\le N$. There are
constructions that show that this constant must grow with~$p$, even if
$\mathcal{K}$ is the class of outerplanar graphs. One such construction
appears in~\cite{NOdM15}. A very simple one, which we sketch in
Figure~\ref{largek}, can be found in~\cite{Q17}.

\begin{figure}[h]
 \centering
 \captionsetup{justification=centering}
 \bigskip
 \includegraphics[height=0.23 in]{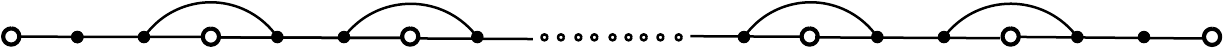}
 \medskip
 \caption{Outerplanar graphs $G$ for which $\omega(G^{odd})$, and hence
   $\chi(G^{odd})$,\newline can be arbitrarily large.}
  \label{largek}
\end{figure}

For a graph $G$, a natural generalisation of
$\edp{G}{1}\cup \edp{G}{3}\cup \dots \cup\edp{G}{p}$ is the
graph~$G^{odd}$, which has the same vertex set as~$G$, and~$xy$ is an edge
in~$G^{odd}$ if and only if $x$ and $y$ have odd distance. Both
constructions in the previous paragraph tell us that for outerplanar
graphs~$G$ the chromatic number of~$G^{odd}$ can be arbitrarily large
because the clique number $\omega(G^{odd})$ can be arbitrarily large. This
motivates the following open problem of Thomass\'{e}, which appears
in~\cite[Section~11.9.3]{NOdM12} (see also \cite{NOdM15}).

\begin{prob}[{\cite[Problem 11.2]{NOdM12}}]\mbox{}\\*
  Is there a function~$f$ such that for every planar graph~$G$ we have
  $\chi(G^{odd})\le f(\omega(G^{odd}))$\,?
\end{prob}

\noindent
Another area that is ripe for further research is the chromatic number of
exact distance graphs with even distance, for specific classes of graphs.
Theorem~\ref{thm10}\,(b) gives a first result for even distances. There is
very little we know about the dependencies between $\chi(\edp{G}{p})$ and
$\wcol_p(G)$ for even~$p$.

It is well-known, and easy to prove (see, e.g., \cite{KMR98}), that for
every graph~$G$ we have $\chi(G^2)\le(2\col(G)-3)\cdot\Delta(G)$, hence
certainly $\chi(\edp{G}{2})\le(2\col(G)-3)\cdot\Delta(G)$. This suggests
that there might exist a function~$\varphi$ such that
$\chi(\edp{G}{p})\le\varphi(\wcol_{p-1}(G))\cdot\Delta(G)$, or even
$\chi(\edp{G}{p})\le\varphi(\wcol_{p/2}(G))\cdot\Delta(G)$. We have not
been able to prove such a result. Neither do we know what the best value of
$r(p)$ should be such that a result of the form
$\chi(\edp{G}{p})\le\varphi(\wcol_{r(p)}(G))\cdot\Delta(G)$ is possible
for even~$p$.

\subsection*{Acknowledgement}

The authors would like to thank the anonymous referees for careful reading
and for their corrections and suggestions.

\end{document}